\documentclass[11pt]{article}
\usepackage[margin=1in]{geometry}
\usepackage{graphicx,xcolor,cite,mathtools}
\usepackage{amssymb,amsmath,amsthm,mathrsfs,paralist,bm,esint,setspace}
\RequirePackage[colorlinks,citecolor=blue,urlcolor=blue]{hyperref}

\newcommand{\R}{{\mathbb{R}}}

\newcommand{\E}{\mathrm{E}}

\renewcommand{\d}{\mathrm{d}}

\newcommand{\e}{\mathrm{e}}

\input cyracc.def

\title{
Gaussian fluctuation  for spatial average of super-Brownian motion}
\author{Zenghu Li\footnote{Research supported  partially by the National Key R\&D Program of China (No. 2020YFA0712900) and the National Science Foundation of China (No. 11531001).
	} \,and  Fei Pu\footnote{Research supported partially by “the Fundamental Research Funds for the Central Universities”.}\\\mbox{} \\Beijing Normal University
	}
\date{\today}

\begin{document}
\newtheorem{stat}{Statement}[section]
\newtheorem{proposition}[stat]{Proposition}
\newtheorem*{prop}{Proposition}
\newtheorem{corollary}[stat]{Corollary}
\newtheorem{theorem}[stat]{Theorem}
\newtheorem{lemma}[stat]{Lemma}
\theoremstyle{definition}
\newtheorem{definition}[stat]{Definition}
\newtheorem*{cremark}{Remark}
\newtheorem{remark}[stat]{Remark}
\newtheorem*{OP}{Open Problem}
\newtheorem{example}[stat]{Example}
\newtheorem{nota}[stat]{Notation}

\numberwithin{equation}{section}
\maketitle

\begin{abstract}
Let $\{u(t\,, x)\}_{(t, x)\in \R_+\times \R}$ be the density of 
one-dimensional super-Brownian motion starting from Lebesgue measure. 
 Using the Laplace functional of super-Brownian motion, we prove that as $N\to \infty$,
 the normalized spatial integral $N^{-1/2}\int_0^{xN}[u(t\,, z)-1 ]\d z$ converges jointly in $(t, x)$ to Brownian sheet in distribution. 
 \end{abstract}

\bigskip

 \smallskip
\noindent{\it Abbreviated title}: CLT for SBM

\section{Introduction and main result}
Let $X_t(\d x)$ be the one-dimensional super-Brownian motion starting from Lebesgue measure. It has a density with respect to Lebesgue measure, that is, $X_t(\d x)= u(t\,, x)\d x$ 
where almost surely $(t, x) \mapsto u(t\,, x)$ is jointly continuous on $(0, \infty)\times \R$ and satisfies the following stochastic partial differential equation (see Konno and Shiga \cite{KoS88})
\begin{align}\label{eq:SBM}
{\partial_t} u(t\,,x)=\frac12 {\partial_x^2}u(t\,, x) + \sqrt{u(t\,,x)}\, \dot{W}(t\,,x),\, t>0, x\in \R
\end{align}
where $\dot{W}$ denotes the space-time white noise.  The solution to the above stochastic heat equation is understood in the weak sense, that is, for any $t>0$ and $f\in C^{\infty}_c(\R)$ (collection of smooth functions with compact support), almost surely, we have 
\begin{align*}
\int_\R f(x)u(t\,, x)\d x = \int_\R f(x)\d x + \int_0^t\int_\R f(x) \sqrt{u(s\,, x)} W(\d s\, \d x) +  \int_0^t\d s \int_\R \frac12 f''(x)u(s\,, x)\d x,
\end{align*}
where $W(\d s\, \d x)$ denotes the stochastic integral with respect to space-time white noise (see Walsh\cite{Walsh}).

The goal of this paper is to establish the following central limit theorem for the solution to \eqref{eq:SBM}.
\begin{theorem}\label{th:main}
          Let $\{u(t\,, x)\}_{(t, x)\in \R_+\times \R}$ be the solution to  \eqref{eq:SBM} with $u(0\,, \cdot)\equiv1$. Then as $N\to \infty$,
          \begin{align*}
          \left\{\frac{1}{\sqrt{N}}\int_0^{xN}[u(t\,, z)- 1]\, \d z\right\}_{(t, x)\in [0, 1]^2} \xrightarrow{ C([0, 1]^2)}
          \left\{W(t\,, x):\, (t, x)\in [0, 1]^2
          \right\},
         \end{align*}
          where $W$ denotes the Brownian sheet and ``$\xrightarrow{C([0,1]^2)}$'' denotes the convergence in distribution in the space of continuous functions $C([0,1]^2)$.       
\end{theorem}

Theorem \ref{th:main} is motivated by the recent progress on the central limit theorem for stochastic partial differential equations; see \cite{HNV2018, CKNP2, CKNP_c} and references therein.  For example, let $\{U(t\,, x)\}_{(t, x)\in \R_+\times \R}$
be the solution to parabolic Anderson model  subject to $U(0)\equiv 1$, driven by space-time white noise. 
Then according to Huang et al \cite[Theorem 1.2]{HNV2018},  for fixed $x\in[0, 1]$, as $N\to\infty$,
          \begin{align*}
         \left\{ \frac{1}{\sqrt{N}}\int_0^{xN}[U(t\,, z) -1] \, \d z\right\}_{t\in [0, 1]}	\xrightarrow{C[0,1]} \left\{\sqrt{x}\int_0^{t}\sqrt{\E[U(s\,, 0)^2]}\, \d {\rm B}_s\right\}_{t\in [0, 1]},
          \end{align*}
          where ${\rm B}$ denotes the standard Brownian motion and ``$\xrightarrow{C[0,1]}$'' denotes the convergence in law in the space of continuous functions $C[0,1]$.  On the other hand,  Chen et al \cite[Theorem 2.3]{CKNP2} have proved that for fixed $t\in [0, 1]$, as $N\to\infty$,
                    \begin{align*}
          \left\{\frac{1}{\sqrt{N}}\int_0^{xN}[U(t\,, z) -1] \, \d z\right\}_{x\in [0, 1]}	\xrightarrow{C[0,1]}
          \{C_t{\rm B}_x: x\in [0, 1]\},
          \end{align*}
          where $C_t^2=\int_0^t \E[U(s\,, 0)^2]\d s$.
          One might expect that as a process in $(t, x)$, the above normalized integral converges jointly to a two-parameter Gaussian process in distribution in the space $C([0, 1]^2)$ as $N\to \infty$; see \cite[Remark 2.5(2)]{CKNP2}. Our Theorem \ref{th:main} provides such a result for the solution to stochastic heat equation \eqref{eq:SBM}.

The SPDEs considered in \cite{HNV2018, CKNP2, CKNP_c} have Lipschitz continuous diffusion coefficient, where we can perform Malliavin calculus with the solution and apply the Poincar\'{e} inequality to obtain the CLT. 
For the stochastic heat equation \eqref{eq:SBM} associated to super-Brownian motion, since the diffusion coefficient is not Lipschitz continuous,  it is not clear if we can use the approaches in \cite{HNV2018, CKNP2, CKNP_c} to study the CLT for the solution. 
In the context of  super-Brownian motion, one can make use of the 
Laplace functional to study the asymptotic behaviors of the process $X_t$ as $t\to \infty$ (see for example \cite{Isc86, Liz99}).
In fact, we will also appeal to  the 
Laplace functional of super-Brownian to prove the central limit theorem for the spatial average of the solution to \eqref{eq:SBM}.

More details on the Laplace functional of super-Brownian motion will be presented in Section \ref{section2}.
 In Sections \ref{finite-dim} and \ref{tightness}, we will establish the convergence of finite dimensional distributions and the tightness respectively. Finally, we prove Theorem \ref{th:main} in Section \ref{th:proof}.

\section{Preliminaries}\label{section2}

We write $\langle \nu\,, f \rangle:= \int_\R f(x)\nu(\d x)$ and $\langle g\,, f \rangle:=\int_\R g(x)f(x)\d x$ for a measure $\nu$ and functions $f$ and $g$. Denote $\mathcal{M}_b(\R)^+$ the collection of nonnegative, bounded and measurable functions.
Recall that the one dimensional super-Brownian motion $X_t(\d x)$ is a measure-valued branching Markov process such that
\begin{align}\label{Laplace}
\E_{X_0}\left[\e^{-\langle X_t\,, f\rangle }\right]= \e^{-\langle X_0\,, V_tf\rangle}, \quad \text{for $f \in \mathcal{M}_b(\R)^+$}.
\end{align}
where the notation $\E_{X_0}$ denotes the super-Brownian motion starting from a finite measure $X_0$,
and the function $(t, x)\mapsto V_t(f)(x)$ is the unique locally bounded and nonnegative solution to the following nonlinear partial differential equation
\begin{align}\label{nonlinearPDE}
\begin{cases}
 {\partial_t}V_t(f)(x)= \frac12{\partial_x^2} V_t(f)(x) -\frac12 [V_t(f)(x)]^2,\\
V_0(f)=f.
\end{cases}
\end{align}
The solution to the above nonlinear PDE satisfies the following integral equation
\begin{align}\label{eq:integral}
V_t(f)(x)= P_tf(x)-\frac12\int_0^tP_{t-s}[V_s(f)]^2(x) \d s,
\end{align}
where  $P_tf(x):=\int_\R\bm{p}_t(x-y)f(y)\d y$ with $\bm{p}_t(x):=(2\pi t)^{-1/2}\e^{-x^2/(2t)}$ for $t>0$ and $x\in \R$.
Furthermore, 
according to \cite[p.54]{Li11} (see also \cite[Lemma 2.1]{Eth00}), we have the following formulas for the  moments of super-Brownian motion: for  $f\in \mathcal{M}_b(\R)^+$,
          \begin{align}
          \E_{X_0}\left[\langle X_t\,, f\rangle\right]&=\langle X_0\,, P_tf\rangle, \label{f:expectation}\\
          \E_{	X_0}\left[\langle X_t\,, f\rangle^2\right]&=
          \langle X_0\,, P_tf\rangle^2+
          \int_0^t\langle X_0\,, P_{t-s}(P_s f)^2\rangle \d s.  
          \label{f:secondmoment}
          \end{align}
We refer to \cite{Daw93, Per02, Eth00, Mue09, Li11} for more information on super-Brownian motion. 

The above facts on the Laplace functional and moments are also true for super-Brownian motion starting from Lebesgue measure  (denoted by $\lambda$); 
see, for example, \cite[Theorem 1.1]{Isc86} and \cite[Theorem 1.4]{KoS88}.  In fact, we can decompose the Lebesgue measure as $\lambda=\sum_{i=1}^{\infty}\lambda_i$, where $\lambda_i$'s are finite
measures on $\R$. By the branching property of super-Brownian motion, the process $X_t$ starting from $\lambda$ is the sum of independent copies of the process starting from $\lambda_i$.
Since each of the summand process satisfies the identities \eqref{Laplace}, \eqref{f:expectation} and \eqref{f:secondmoment}, by independence, it implies that $X_t$ starting from $\lambda$ also
satisfies theses properties. Denote $L_b^1(\R)^+ = L^1(\R)\cap \mathcal{M}_b(\R)^+$. Then, we have  for  $f \in L_b^1(\R)^+$,
          \begin{align}
          \E_{\lambda}\left[\langle X_t\,, f\rangle\right]&=\langle \lambda\,, f\rangle, \label{expectation}\\
          \E_{	\lambda}\left[\langle X_t\,, f\rangle^2\right]&=
          \langle \lambda\,, f\rangle^2+
          \int_0^t\langle \lambda\,, (P_s f)^2\rangle \d s.  
          \label{secondmoment}
          \end{align}
          Moreover, 
we see from \eqref{eq:integral} that for $f \in L_b^1(\R)^+$,
\begin{align}\label{inner}
\langle \lambda\,, V_t(f)\rangle = \langle \lambda\,, f\rangle - \frac12\int_0^t\int_\R [V_s(f)(y)]^2\d y \d s.
\end{align}



For $f\in L^1(\R)$, we denote $\hat{f}$ the Fourier transform of $f$, that is, 
\begin{align*}
\hat{f}(x)= \int_\R f(y)\e^{ixy}\d y, \quad x\in \R.
\end{align*}
Moreover, we introduce \text{for $N>0$ and $f \in L^1(\R)$}
\begin{align}\label{scale}
f^{(N)}(x):= \frac{1}{\sqrt{N}}f(x/N), \,\, x\in \R.
\end{align}
The following technical lemma will be used later on. 
\begin{lemma}\label{Plancherel}
          For all $t>0$, $r_1, r_2\geq 0$ and $f, g\in L^1(\R)\cap L^2(\R)$,
          \begin{align}\label{eq:asym}
          \lim_{N\to\infty}\int_0^{t}\d s\int_\R\d y\, 
          P_{s+r_1}f^{(N)}(y)P_{s+r_2}g^{(N)}(y)= t\cdot \langle f\,, g\rangle,
          \end{align}
          and 
          \begin{align}\label{eq:asym2}
          \sup_{N>0}\int_0^{t}\d s\int_\R\d y\, 
          \left[P_{s+r_1}f^{(N)}(y)\right]^2\leq t\cdot \langle f\,, f\rangle.
          \end{align}
\end{lemma}
\begin{proof}
          By the semigroup property of heat kernel, we write
          \begin{align*}
          &\int_\R\d y\, 
          P_{s+r_1}f^{(N)}(y)P_{s+r_2}g^{(N)}(y)\\
           &\quad=       \int_\R\d y\,
          \int_\R \bm{p}_{s+r_1}(y-z_1)f^{(N)}(z_1) \d z_1          \int_\R \bm{p}_{s+r_2}(y-z_2)g^{(N)}(z_2)\d z_2\\
          &\quad=         
          \int_{\R^2} \bm{p}_{2s+r_1+r_2}(z_1-z_2)f^{(N)}(z_1)g^{(N)}(z_2)\d z_1\d z_2\\
          &\quad=\left(\bm{p}_{2s+r_1+r_2}*f^{(N)}*\widetilde{g^{(N)}}\right)(0),
          \end{align*}
          where $\tilde{g}(x):=g(-x)$ for $x\in \R$.
          Appealing to the Plancherel's identity, 
          \begin{align}\label{id:Plan}
          \left(\bm{p}_{2s+r_1+r_2}*f^{(N)}*\widetilde{g^{(N)}}\right)(0)
                    &=\frac{1}{2\pi }\int_{\R} \e^{-(s+(r_1+r_2)/2)z^2}
          \widehat{f^{(N)}}(z)\overline{\widehat{g^{(N)}}}(z)\d z\nonumber\\
          &=\frac{N}{2\pi }\int_{\R} \e^{-(s+(r_1+r_2)/2)z^2}
          \hat{f}(Nz)\overline{\hat{g}}(Nz)\d z.
          \end{align}
          Therefore, we obtain that 
                    \begin{align*}
          \int_0^{t}\d s\int_\R\d y\, 
          P_{s+r_1}f^{(N)}(y)P_{s+r_2}g^{(N)}(y) 
          &=\frac{1}{2\pi }\int_0^t\int_{\R} \e^{-(s+(r_1+r_2)/2)z^2/N^2}
          \hat{f}(z)\overline{\hat{g}}(z)\d z \d s,
          \end{align*}
          which implies \eqref{eq:asym} by dominated convergence theorem. Letting $f=g$ and $r_1=r_2$ in
           the proceeding display, we obtain 
           \eqref{eq:asym2}.
\end{proof}

Let us close this section with a brief description of some other notation of this
paper. Throughout we write ``$g_1(x)\lesssim g_2(x)$ for all $x\in \R$'' when
there exists a real number $L$ such that $g_1(x)\le Lg_2(x)$ for all $x\in \R$.
Alternatively, we might write ``$g_2(x)\gtrsim g_1(x)$ for all $x\in \R$.'' By
``$g_1(x)\asymp g_2(x)$ for all $x\in \R$'' we mean that $g_1(x)\lesssim g_2(x)$
for all $x\in \R$ and $g_2(x)\lesssim g_1(x)$ for all $x\in \R$.  We denote $\|g\|_{\infty}:=\sup_{x\in\R}|g(x)|$
for a bounded function $g$. Finally, we write $\|X\|_{k}=(\E[|X|^k])^{1/k}$ for a random variable $X$ and $k\geq 1$.

\section{Convergence of finite-dimensional distributions} \label{finite-dim}

Denote $\{B_{t}(f): t\geq 0, f\in L^2(\R)\}$ as the cylindrical Brownian motion, which is a centred Gaussian process such that
\begin{align*}
\E[B_t(f)B_s(g)] = (t\wedge s)\cdot \langle f\,, g\rangle, \quad  \text{for $t,s \geq 0$ and $f, g\in L^2(\R)$}.
\end{align*}
We have the following result on the convergence of  Laplace transform of super-Brownian motion.

\begin{proposition}\label{fddcon}
          For $0<t_m< \ldots < t_1$ and $f_1\ldots, f_m\in L_b^1(\R)^+$, 
          \begin{align}\label{eq:fdd}
          \lim_{N\to\infty}\E_{\lambda}\left[\e^{-\sum_{k=1}^m
           \left(\langle X_{t_k}\,, f^{(N)}_k\rangle-\E\left[\langle X_{t_k}\,, f^{(N)}_k\rangle\right] \right) }\right]=
           \E\left[\e^{-\sum_{k=1}^m B_{t_k}(f_k)}\right],
          \end{align}
          where the functions $f_k^{(N)}$, $k=1, \ldots, m$, are defined as in \eqref{scale}.
\end{proposition}
\begin{proof}
          By the Markov property and the identity \eqref{Laplace}, we write
          \begin{align}\label{Markov}
          \E_{\lambda}\left[\e^{-\sum_{k=1}^m
           \left(\langle X_{t_k}\,, f^{(N)}_k\rangle\right) }\right]=\e^{- \langle \lambda\,, V_{t_m}(F_m)\rangle}, 
          \end{align}
          where the function $F_m$ is defined as 
          \begin{align}\label{F}
          \begin{cases}
         F_1=f^{(N)}_1,\\
         F_k=f^{(N)}_k+ V_{t_{k-1}-t_k}(F_{k-1}), \quad \text{for $k=2, \ldots, m$}.
        \end{cases}
        \end{align}
        Since $V_t(f)(x) \leq P_tf(x)$ for all $x\in \R$ and nonnegative and bounded function $f$,
        \begin{align}
        \|F_k\|_{\infty}& \leq \|f_k^{(N)}\|_{\infty} + \|P_{t_{k-1}-t_k}F_{k-1}\|_{\infty}\nonumber\\
        &\leq  \|f_k^{(N)}\|_{\infty} + \|F_{k-1}\|_{\infty},\nonumber
        \end{align}
        which implies that 
        \begin{align}\label{bound1}
        \|F_k\|_{\infty} \leq \sum_{j=1}^k\|f_j^{(N)}\|_{\infty} \leq \frac{1}{\sqrt{N}}\sum_{j=1}^k\|f_j\|_{\infty}.
        \end{align}
        Moreover,
        \begin{align}
        \langle \lambda\,, F_k\rangle &\leq         \langle \lambda\,, f^{(N)}_k\rangle + 
        \langle \lambda\,, P_{t_{k-1}-t_k}F_{k-1}\rangle\nonumber\\
        &=         \langle \lambda\,, f^{(N)}_k\rangle + 
        \langle \lambda\,, F_{k-1}\rangle,\nonumber
        \end{align}
        which implies that
        \begin{align}\label{bound2}
        \langle \lambda\,, F_k\rangle &\leq \sum_{j=1}^k\langle \lambda\,, f^{(N)}_j\rangle
         = \sqrt{N}\sum_{j=1}^k\langle \lambda\,, f_j\rangle.
        \end{align}

        Set $t_{m+1}=0$.
        Using \eqref{inner} and \eqref{F},
        \begin{align}\label{iteration}
        \langle \lambda\,, V_{t_m-t_{m+1}}(F_m)\rangle&= \langle \lambda\,, F_m\rangle -
         \frac12\int_0^{t_m-t_{m+1}}\int_\R [V_s(F_m)(y)]^2\d y \d s \nonumber\\
         &=\langle\lambda\,, f_m^{(N)}\rangle +    \langle \lambda\,, V_{t_{m-1}-t_m}(F_{m-1})\rangle-
         \frac12\int_0^{t_m-t_{m+1}}\int_\R [V_s(F_m)(y)]^2\d y \d s \nonumber\\
         &\, \,\,\vdots \nonumber\\
         &=\sum_{k=1}^m\langle\lambda\,, f_k^{(N)}\rangle 
         -\sum_{k=1}^m\frac12\int_0^{t_k-t_{k+1}}\int_\R [V_s(F_k)(y)]^2\d y \d s.
        \end{align}
        Denote
        \begin{align}\label{K=I_k}
        I_k^{(N)}=\int_0^{t_k-t_{k+1}}\int_\R [V_s(F_k)(y)]^2\d y \d s.
        \end{align}
        Hence, we see from \eqref{Markov}, \eqref{iteration} and \eqref{expectation} that
        \begin{align}\label{eq:fdd2}
          \E_{\lambda}\left[\e^{-\sum_{k=1}^m
           \left(\langle X_{t_k}\,, f^{(N)}_k\rangle-\E\left[\langle X_{t_k}\,, f^{(N)}_k\rangle\right] \right) }\right]=
           \exp\left\{
           \frac12\sum_{k=1}^mI_{k}^{(N)}
           \right\}.
          \end{align}

          We next estimate $I_{k}^{(N)}$. We apply the formula \eqref{eq:integral} to write
          \begin{align}
          I_{k}^{(N)} &=\int_0^{t_k-t_{k+1}}\d s\int_\R\d y\, \left[
          P_sF_k(y) - \frac12 \int_0^sP_{s-r}[V_rF_k]^2(y)\d r
          \right]^2\nonumber\\
          &=I^{(N)}_{k,1}+I^{(N)}_{k,2} + I^{(N)}_{k,3},
          \end{align}
          where 
          \begin{align*}
          I^{(N)}_{k,1}&=\int_0^{t_k-t_{k+1}}\d s\int_\R\d y\, \left[
          P_sF_k(y) 
          \right]^2,\\
          I^{(N)}_{k,2}&=-\int_0^{t_k-t_{k+1}}\d s\int_\R\d y\, 
          P_sF_k(y)  \int_0^sP_{s-r}[V_rF_k]^2(y)\d r
          ,\\
          I^{(N)}_{k,3}&=\frac14\int_0^{t_k-t_{k+1}}\d s\int_\R\d y\, \left[
           \int_0^sP_{s-r}[V_rF_k]^2(y)\d r 
          \right]^2.
          \end{align*}
          By \eqref{bound1} and \eqref{bound2}, 
          \begin{align}\label{estimate:I2}
          \left|I^{(N)}_{k,2}\right| &\leq  \int_0^{t_k-t_{k+1}}\d s\int_\R\d y\, 
          P_sF_k(y)  \int_0^s\|V_rF_k\|_{\infty}^2\d r\nonumber\\
          &\leq  \int_0^{t_k-t_{k+1}}\d s\int_\R\d y\, 
          P_sF_k(y)  \int_0^s\|F_k\|_{\infty}^2\d r
           =\frac{(t_k-t_{k+1})^2}{2}\|F_k\|_{\infty}^2 \langle \lambda\,, F_k\rangle\nonumber\\
           & \leq \frac{1}{\sqrt{N}} \frac{(t_k-t_{k+1})^2}{2}\sum_{j=1}^{k}\langle \lambda\,, f_j\rangle 
           \left(\sum_{j=1}^{k}\|f_j\|_{\infty}\right)^2.
          \end{align}
          Similarly, 
          \begin{align}\label{estimate:I3}
          I^{(N)}_{k,3} &\leq\frac{t_k-t_{k+1}}4 \|F_k\|_{\infty}^3\int_0^{t_k-t_{k+1}}\d s\int_\R\d y\, 
           \int_0^sP_{s-r}V_rF_k(y)\d r \nonumber\\
           &\leq\frac{t_k-t_{k+1}}4  \|F_k\|_{\infty}^3\int_0^{t_k-t_{k+1}}\d s\int_\R\d y\, 
           \int_0^sP_{s-r}P_rF_k(y)\d r \nonumber\\
           &= \|F_k\|_{\infty}^3\frac{(t_k-t_{k+1})^3}{8}\langle \lambda\,, F_k\rangle \nonumber\\
           & \leq \frac{1}{N} \frac{(t_k-t_{k+1})^3}{8}\sum_{j=1}^{k}\langle \lambda\,, f_j\rangle 
           \left(\sum_{j=1}^{k}\|f_j\|_{\infty}\right)^3.
          \end{align}
          The estimates in \eqref{estimate:I2} and \eqref{estimate:I3} yield that as $N\to\infty$,
          \begin{align}\label{1}
          I_{k}^{(N)} =           I^{(N)}_{k,1} + o(1).
          \end{align}
          Furthermore, appealing to the identity \eqref{F},
          \begin{align}\label{2}
          I^{(N)}_{k,1}&= \int_0^{t_k-t_{k+1}}\d s\int_\R\d y\, \left[
          P_sf^{(N)}_k(y) + P_sV_{t_{k-1}-t_{k}}(F_{k-1})(y)
          \right]^2\nonumber\\
          &=          I^{(N)}_{k,1,1}+           I^{(N)}_{k,1,2}+           I^{(N)}_{k,1,3},
          \end{align}
          where
          \begin{align*}
          I^{(N)}_{k,1,1}&=\int_0^{t_k-t_{k+1}}\d s\int_\R\d y\, \left[
          P_sf^{(N)}_k(y)
          \right]^2,\\
          I^{(N)}_{k,1,2}&=2\int_0^{t_k-t_{k+1}}\d s\int_\R\d y\, 
          P_sf^{(N)}_k(y) P_sV_{t_{k-1}-t_{k}}(F_{k-1})(y),\\
         I^{(N)}_{k,1,3}&=\int_0^{t_k-t_{k+1}}\d s\int_\R\d y\, \left[
           P_sV_{t_{k-1}-t_{k}}(F_{k-1})(y) 
          \right]^2.
          \end{align*}
          By Lemma \ref{Plancherel}, as $N\to\infty$,
          \begin{align}\label{3}
          I^{(N)}_{k,1,1} = (t_k-t_{k+1})\cdot \langle f_k\,, f_k\rangle + o(1).
          \end{align}
         Moreover, we use \eqref{eq:integral} to write
         \begin{align}
         I^{(N)}_{k,1,2}&=2\int_0^{t_k-t_{k+1}}\d s\int_\R\d y\, 
          P_sf^{(N)}_k(y) P_sP_{t_{k-1}-t_{k}}(F_{k-1})(y)\nonumber\\
          & \quad -\int_0^{t_k-t_{k+1}}\d s\int_\R\d y\, P_sf^{(N)}_k(y) \int_0^{t_{k-1}-t_k}
          P_sP_{t_{k-1}-t_k-r}[V_{r}(F_{k-1})]^2(y)\d r \nonumber\\
          &=2\int_0^{t_k-t_{k+1}}\d s\int_\R\d y\, 
          P_sf^{(N)}_k(y) P_{t_{k-1}-t_{k}+s}(F_{k-1})(y) + o(1), \quad \text{as $N\to\infty$},\nonumber
         \end{align}
         where the second equality holds by \eqref{bound1} and \eqref{bound2}.
         Using the formula \eqref{F}, we have as $N\to\infty$
         \begin{align}
         I^{(N)}_{k,1,2}&=2\int_0^{t_k-t_{k+1}}\d s\int_\R\d y\, 
          P_sf^{(N)}_k(y) P_{t_{k-1}-t_{k}+s}(f^{(N)}_{k-1})(y) \nonumber\\
          &\quad + 2\int_0^{t_k-t_{k+1}}\d s\int_\R\d y\, 
          P_sf^{(N)}_k(y) P_{t_{k-1}-t_{k}+s}(V_{t_{k-2}-t_{k-1}}(F_{k-2}))(y) + o(1)\nonumber\\
          &= 2(t_k-t_{k+1})\cdot \langle f_k\,, f_{k-1}\rangle + 2\int_0^{t_k-t_{k+1}}\d s\int_\R\d y\, 
          P_sf^{(N)}_k(y) P_{t_{k-1}-t_{k}+s}(V_{t_{k-2}-t_{k-1}}(F_{k-2}))(y)\nonumber\\
          & \quad + o(1),\nonumber
         \end{align}
         thanks to Lemma \ref{Plancherel}.
         Since the estimates in \eqref{bound1} and \eqref{bound2} ensure that
          the dominated term of $V_{t_{k-2}-t_{k-1}}(F_{k-2})$ is  $P_{t_{k-2}-t_{k-1}}(F_{k-2})$, 
           we can repeat the proceeding 
         argument and apply Lemma \ref{Plancherel} to 
         conclude that as $N\to\infty$,
         \begin{align}\label{4}
         I^{(N)}_{k,1,2}&=2(t_k-t_{k+1})\sum_{j=1}^{k-1} \langle f_k\,, f_{j}\rangle + o(1).
         \end{align}

         As for $I_{k,1,3}^{(N)}$, we can use the same argument in \eqref{estimate:I2} and \eqref{estimate:I3} to see 
         that as $N\to\infty$, 
         \begin{align}\label{5}
         I_{k,1,3}^{(N)} =\int_0^{t_k-t_{k+1}}\d s\int_\R\d y\, \left[
           P_{t_{k-1}-t_{k}+s}F_{k-1}(y) 
          \right]^2 + o(1).
         \end{align}
         Therefore, we combine \eqref{1}-\eqref{4} to obtain that as $N\to\infty$, 
         \begin{align}\label{sum-iteration}
         I_{k}^{(N)}&= (t_k-t_{k+1}) \left(\langle f_k\,, f_k\rangle +
         2\sum_{i=1}^{k-1}\langle f_k\,, f_i\rangle \right)+ o(1)\nonumber\\
         &\quad + \int_0^{t_k-t_{k+1}}\d s\int_\R\d y\, \left[
           P_{t_{k-1}-t_{k}+s}F_{k-1}(y) 
          \right]^2.
         \end{align}
         The estimate of the above double integral is similar to that of $I_{k,1}^{(N)}$, 
         with $F_k$ replaced by $F_{k-1}$. We proceed as follows. 
         Using again the formula \eqref{F}, we obtain
         \begin{align}
         &\int_0^{t_k-t_{k+1}}\d s\int_\R\d y\, \left[
           P_{t_{k-1}-t_{k}+s}F_{k-1}(y) 
          \right]^2 \nonumber\\
          &\quad =\int_0^{t_k-t_{k+1}}\d s\int_\R\d y\, \left[
           P_{t_{k-1}-t_{k}+s}f^{(N)}_{k-1}(y) 
           + P_{t_{k-1}-t_{k}+s}V_{t_{k-2}-t_{k-1}}(F_{k-2})(y)
          \right]^2\nonumber\\
          &\quad = J^{(N)}_{k, 1} +J^{(N)}_{k, 2}+ J^{(N)}_{k, 3},
         \end{align}
         where
         \begin{align}\label{J123}
         J^{(N)}_{k, 1}&=\int_0^{t_k-t_{k+1}}\d s\int_\R\d y\, \left[
           P_{t_{k-1}-t_{k}+s}f^{(N)}_{k-1}(y) 
          \right]^2\nonumber\\
          &=(t_k-t_{k+1})\cdot \langle f_{k-1}\,, f_{k-1}\rangle + o(1),\\
                  J^{(N)}_{k, 2}&=2\int_0^{t_k-t_{k+1}}\d s\int_\R\d y\, 
           P_{t_{k-1}-t_{k}+s}f^{(N)}_{k-1}(y) 
           P_{t_{k-1}-t_{k}+s}V_{t_{k-2}-t_{k-1}}(F_{k-2})(y)\nonumber\\
           &=2\int_0^{t_k-t_{k+1}}\d s\int_\R\d y\, 
           P_{t_{k-1}-t_{k}+s}f^{(N)}_{k-1}(y) 
           P_{t_{k-1}-t_{k}+s}P_{t_{k-2}-t_{k-1}}F_{k-2}(y) + o(1),\nonumber\\
           J^{(N)}_{k, 3}&=\int_0^{t_k-t_{k+1}}\d s\int_\R\d y\, \left[
            P_{t_{k-1}-t_{k}+s}V_{t_{k-2}-t_{k-1}}(F_{k-2})(y)
          \right]^2\nonumber\\
          &=\int_0^{t_k-t_{k+1}}\d s\int_\R\d y\, \left[
            P_{t_{k-1}-t_{k}+s}P_{t_{k-2}-t_{k-1}}F_{k-2}(y)
          \right]^2 +o(1) \nonumber\\
          &=\int_0^{t_k-t_{k+1}}\d s\int_\R\d y\, \left[
            P_{t_{k-2}-t_{k}+s}F_{k-2}(y)
          \right]^2 +o(1), \label{J3}
         \end{align}
         as $N\to\infty$ by Lemma \ref{Plancherel} and \eqref{bound1}, \eqref{bound2}.
         Similar to the estimate of $I_{k,1,2}^{(N)}$ in \eqref{4}, we have as $N\to\infty$,
         \begin{align}\label{J2}
         J_{k,2}^{(N)}= 2(t_{k}-t_{k+1}) \sum_{i=1}^{k-2}\langle f_{k-1}\,, f_{i}\rangle + o(1).
         \end{align}
         Now we see from \eqref{sum-iteration}-\eqref{J2} that as $N\to\infty$
                  \begin{align*}
         I_{k}^{(N)}&= (t_k-t_{k+1})\sum_{j=k-1}^k \left(\langle f_j\,, f_j\rangle +
         2\sum_{i=1}^{j-1}\langle f_j\,, f_i\rangle \right)+ o(1)\nonumber\\
         &\quad + \int_0^{t_k-t_{k+1}}\d s\int_\R\d y\, \left[
            P_{t_{k-2}-t_{k}+s}F_{k-2}(y)
          \right]^2.
         \end{align*}
         We can repeat the above argument to conclude that as $N\to\infty$,
         \begin{align*}
         I_{k}^{(N)}&= (t_k-t_{k+1})\sum_{j=1}^k \left(\langle f_j\,, f_j\rangle +
         2\sum_{i=1}^{j-1}\langle f_j\,, f_i\rangle \right)+ o(1).
                  \end{align*}
         Therefore, as $N\to\infty$,
         \begin{align}
         \sum_{k=1}^mI_{k}^{(N)}&= \sum_{k=1}^m(t_k-t_{k+1})\sum_{j=1}^k \left(\langle f_j\,, f_j\rangle +
         2\sum_{i=1}^{j-1}\langle f_j\,, f_i\rangle \right) + o(1)\nonumber\\
         &=\sum_{j=1}^m \left(\langle f_j\,, f_j\rangle +
         2\sum_{i=1}^{j-1}\langle f_j\,, f_i\rangle \right)\sum_{k=j}^m(t_k-t_{k+1}) + o(1)\nonumber\\
         &=\sum_{j=1}^m t_j\left(\langle f_j\,, f_j\rangle +
         2\sum_{i=1}^{j-1}\langle f_j\,, f_i\rangle \right) +o(1),
         \end{align}
         which together with \eqref{eq:fdd2} implies \eqref{eq:fdd}. The proof is complete. 
 \end{proof}

\begin{remark}\label{higherdimension}
The result in Proposition \ref{fddcon}  also holds for $d$-dimensional super-Brownian motion starting from Lebesgue measure on $\R^d$, with   the scaled function in \eqref{scale} defined as $f^{(N)}(x)=N^{-d/2}f(x/N), x\in \R^d$ ($f\in L^1_b(\R^d)$) and the limit process in Proposition \ref{fddcon} replaced by the cylindrical Brownian motion  
$\{B_t(f): t\geq 0, f\in L^2(\R^d)\}$ on $\R^d$. The proof follows along the same lines as in Proposition \ref{fddcon} using an analogue of Lemma  \ref{Plancherel}. 
\end{remark}

\section{Tightness} \label{tightness}

We first recall from   \cite[Lemma 2.5]{KoS88} that 
           for $t>0$ and  $\phi\in C_c^{\infty}(\R)$, almost surely,  
          \begin{align}\label{mild}
          \langle X_t\,, \phi\rangle - \langle \lambda\,, P_t\phi\rangle= \int_0^t\int_\R \sqrt{u(r\,, z)}P_{t-r}\phi(z) W(\d r\, \d z).
          \end{align}
Using stochastic Fubini's theorem, it implies that for all $t>0$ and $x\in \R$, almost surely, 
\begin{align}\label{eq:mild}
          u(t\,, x)= 1+ \int_0^t\int_\R \bm{p}_{t-s}(x-y) \sqrt{u(s\,, y)}\  W(\d s\, \d y).
          \end{align}

\begin{lemma}\label{lem:holder}
          Let $u(0, \cdot)\equiv 1$. For all $T \geq 0$ and $k\geq 2$, there exists a constant $C>0$ such that for all $(t, x), (s, y)\in [0, T]\times \R$
          \begin{align}\label{eq:holder}
          \E\left[|u(t\,, x)-u(s\,, y)|^{k}\right] \leq C \left(|t-s|^{k/4} + |x-y|^{k/2}\right).
          \end{align}
          As a consequence, the solution to \eqref{eq:SBM} subject to $u(0, \cdot)\equiv 1$ admits a continuous version on $\R_+ \times \R$.
\end{lemma}
\begin{proof}
          We first notice from the mild form \eqref{eq:mild} that for all $t>0$, $\{u(t\,, x): x\in \R\}$ is a stationary process, which together with \cite[Lemma 2.7]{KoS88} implies that 
          \begin{align}\label{momentbound}
          \sup_{t\in [0, T] \times \R}\E[u(t\,,x)^k] <\infty.
          \end{align}
          Without loss of generality, we assume $t\geq s\geq 0$. Then by Burkholder's inequality,
          \begin{align}\label{momentsum}
          \E\left[|u(t\,, x)-u(s\,, y)|^{k}\right] &\lesssim \E\left[\left|\int_s^t\int_\R \bm{p}_{t-r}(x-z)\sqrt{u(r\,, z)}\, W(\d r\, \d z)\right|^k\right] \nonumber\\
          & \quad+ \E\left[\left|\int_0^s\int_\R \left(\bm{p}_{t-r}(x-z)- \bm{p}_{s-r}(y-z)\right)\sqrt{u(r\,, z)}\, W(\d r\, \d z)\right|^k\right] \nonumber\\
          & \lesssim  \E\left[\left|\int_s^t\int_\R \bm{p}^2_{t-r}(x-z)u(r\,, z)\, \d z\d r\right|^{k/2}\right] \nonumber\\
          &\quad + \E\left[\left|\int_0^s \int_\R\left(\bm{p}_{t-r}(x-z)- \bm{p}_{s-r}(y-z)\right)^2u(r\,, z)\, \d z\d r\right|^{k/2}\right] 
          \end{align}
          Using Minkowski's inequality and \eqref{momentbound}, 
          \begin{align}\label{sum1}
          \E\left[\left|\int_s^t \int_\R\bm{p}^2_{t-r}(x-z)u(r\,, z)\, \d z\d r\right|^{k/2}\right] &\leq \left(\int_s^t\int_\R \bm{p}^2_{t-r}(x-z)\|u(r\,, z)\|_{k/2}\, \d z\d r\right)^{k/2}\nonumber\\
          & \lesssim \left(\int_s^t\int_\R \bm{p}^2_{t-r}(x-z)\d z\d r\right)^{k/2} \asymp (t-s)^{k/4},
          \end{align}
          thanks to the semigroup property of heat kernel.
          Similarly, 
          \begin{align}\label{sum2}
          &\E\left[\left|\int_0^s \int_\R\left(\bm{p}_{t-r}(x-z)- \bm{p}_{s-r}(y-z)\right)^2u(r\,, z)\, \d z\d r\right|^{k/2}\right]\nonumber\\
          & \quad \lesssim \left(\int_0^s \int_\R\left(\bm{p}_{t-r}(x-z)- \bm{p}_{s-r}(y-z)\right)^2\|u(r\,, z)\|_{k/2}\, \d z\d r\right)^{k/2} \nonumber\\
          & \quad \lesssim  \left(\int_0^s \int_\R\left(\bm{p}_{t-r}(x-z)- \bm{p}_{s-r}(y-z)\right)^2\d z\d r\right)^{k/2} \nonumber\\
          & \quad \lesssim |t-s|^{k/4} + |x-y|^{k/2},
          \end{align}
          where the last inequality follows from \cite[Proposition 5.2]{ChD14}.  Therefore, we combine \eqref{sum1} and \eqref{sum2} to obtain \eqref{eq:holder}.
          Finally, by Kolmogorov continuity theorem, $\{u(t\,, x): (t, x)\in \R_+\times \R\}$ has a continuous version on  $\R_+ \times \R$.
\end{proof}

Recall that for the indicate function $\bm{1}_{[0, x]}$ ($x\geq 0$), the function $\bm{1}^{(N)}_{[0, x]}$is defined as in \eqref{scale}.                
For $N\geq 1$ and $(t, x)\in [0,1]^2$, we introduce
\begin{align}\label{V_N}
V_N(t\,, x)&:= \frac1{\sqrt{N}}\int_0^{xN}[u(t\,, z) -1]\d z \nonumber\\
&= \langle X_t\,, \bm{1}_{[0,\, x]}^{(N)}\rangle - \E\left[\langle X_t\,, \bm{1}_{[0,\, x]}^{(N)}\rangle\right]\nonumber\\
&=\int_0^t\int_\R \sqrt{u(r\,, z)} P_{t-r}\bm{1}_{[0,\, x]}^{(N)}(z)W(\d r\, \d z),
\end{align}
where the last equality holds by \eqref{eq:mild} and stochastic Fubini's theorem. 
Moreover, it is clear from Lemma \ref{lem:holder} that for each $N\geq1$, the process $\{V_N(t\,,x): (t,x)\in [0, 1]^2\}$ is continuous on $[0, 1]^2$.

\begin{proposition}\label{prop:holder}
          There exists a constant $C>0$ such that for all $N\geq 1$ and $(t, x)$, $(s, y)\in [0, 1]^2$,
          \begin{align}\label{eq:holder}
          \E_{\lambda}\left[\left| V_N(t\,,x)+ V_N(s\,,y)- V_N(t\,, y)-V_N(s\,,x)
          \right|^4 \right] \leq C\left(|t-s| \times |x-y|\right)^{5/4}.
          \end{align}
\end{proposition}
\begin{proof}
                     Assume $t\geq s$ and $x\geq y$ without loss of generality.
          According to \eqref{V_N}, we write
          \begin{align*}
           &V_N(t\,,x)+ V_N(s\,,y)- V_N(t\,, y)-V_N(s\,,x)\\
           &\quad =\int_s^t\int_\R \sqrt{u(r\,, z)} P_{t-r}\bm{1}_{[y,\, x]}^{(N)}(z)W(\d r\, \d z)\\
           & \quad \quad +
             \int_0^s\int_\R\sqrt{u(r\,, z)} \left(P_{t-r}\bm{1}_{[y,\, x]}^{(N)}- P_{s-r}\bm{1}_{[y,\, x]}^{(N)}\right)(z)W(\d r\, \d z).
          \end{align*}
          Hence, 
          \begin{align}\label{A+B}
          &\E_{\lambda}\left[\left| V_N(t\,,x)+ V_N(s\,,y)- V_N(t\,, y)-V_N(s\,,x)
          \right|^4\right]\leq 8(\mathcal{A}+ \mathcal{B}),
          \end{align}
          where
          \begin{align*}
          \mathcal{A}&=\E_{\lambda}\left[\left|\int_s^t\int_\R \sqrt{u(r\,, z)} P_{t-r}\bm{1}_{[y,\, x]}^{(N)}(z)W(\d r\, \d z)\right|^4\right],\\
          \mathcal{B}&=\E_{\lambda}\left[\left|\int_0^s\int_\R\sqrt{u(r\,, z)} \left(P_{t-r}\bm{1}_{[y\,, x]}^{(N)}- P_{s-r}\bm{1}_{[y\,, x]}^{(N)}
          \right)(z)W(\d r\, \d z)
          \right|^4\right].
          \end{align*}
          Applying the Burkholder's inequality and the Cauchy-Schwartz inequality, we have 
          \begin{align}\label{Abound}
          \mathcal{A} &\leq 4 \E_{\lambda}\left[\left| \int_s^t \left\langle X_r\,, \left[P_{t-r}\bm{1}_{[y,\, x]}^{(N)}\right]^2\right
          \rangle\d r\right|^2
           \right]\nonumber\\
           &\leq 4(t-s) \int_s^t  \E_{\lambda}\left[\left\langle X_r\,, \left[P_{t-r}\bm{1}_{[y,\, x]}^{(N)}\right]^2\right\rangle^2 \right]\d r
           \nonumber\\
           &= 4(t-s) \int_s^t \d r\, \left(
                       \left\langle \lambda\,, \left[P_{t-r}\bm{1}_{[y,\, x]}^{(N)}\right]^2\right\rangle^2+
           \int_0^r\left\langle \lambda\,, \left(P_\theta\left[P_{t-r}\bm{1}_{[y,\, x]}^{(N)}\right]^2\right)^2\right\rangle\d \theta
           \right),
          \end{align}
          where the equality follows from \eqref{secondmoment}.
          The formula \eqref{id:Plan} ensures that
          \begin{align}\label{x-y}
           \left\langle \lambda\,, \left[P_{t-r}\bm{1}_{[y,\, x]}^{(N)}\right]^2\right\rangle
           &=\frac{N}{2\pi}\int_\R \e^{-(t-r)z^2}\left|\widehat{\bm{1}_{[y\,, x]}}(Nz)\right|^2\d z\nonumber\\
           &\leq \frac{N}{2\pi}\int_\R \left|\widehat{\bm{1}_{[y\,, x]}}(Nz)\right|^2\d z = \|\bm{1}_{[y\,, x]}\|^2_{L^2(\R)}=|x-y|,
          \end{align}
          which implies that
          \begin{align}\label{A2}
          (t-s)\int_s^t\d r\,            \left\langle \lambda\,, \left[P_{t-r}\bm{1}_{[y,\, x]}^{(N)}\right]^2\right\rangle^2
           \leq |t-s|^2|x-y|^2.
          \end{align}
          Denote $F=\left[P_{t-r}\bm{1}_{[y,\, x]}^{(N)}\right]^2$.
          By the Plancherel's identity (see the calculation in  \eqref{id:Plan}), 
          \begin{align*}
          \left\langle \lambda\,, \left(P_\theta\left[P_{t-r}\bm{1}_{[y,\, x]}^{(N)}\right]^2\right)^2\right\rangle
          &=   \left\langle \lambda\,, \left(P_\theta F\right)^2\right\rangle
          =\frac{1}{2\pi}\int_\R\e^{-\theta z^2}|\hat{F}(z)|^2\d z.
          \end{align*}
          Note that for all $z\in \R$
          \begin{align*}
          |\hat{F}(z)|\leq \langle \lambda\,, F\rangle &=\left\langle \lambda\,,
           \left[P_{t-r}\bm{1}_{[y,\, x]}^{(N)}\right]^2\right\rangle \leq |x-y|,
          \end{align*}
          where the last inequality holds by \eqref{x-y}. The proceeding yields that
          \begin{align}\label{A1}
          &(t-s)\int_s^t \d r\, 
           \int_0^r\left\langle \lambda\,, \left(P_\theta\left[P_{t-r}\bm{1}_{[y,\, x]}^{(N)}\right]^2\right)^2\right\rangle\d \theta
           \nonumber\\
           &\quad \qquad\qquad\qquad\qquad\leq \frac{(t-s)|x-y|^2}{2\pi}\int_s^t\d r\, \int_0^r\d \theta\, \int_\R\e^{-\theta z^2}\d z\nonumber\\
           &\quad\qquad\qquad\qquad\qquad \asymp  (t-s)|x-y|^2\int_s^t \sqrt{r}\d r  \lesssim |t-s|^2|x-y|^2.
          \end{align}
          Hence, we conclude from \eqref{Abound},  \eqref{A2} and \eqref{A1} that 
          \begin{align}\label{Abound2}
          \mathcal{A}\lesssim |t-s|^2|x-y|^2. 
          \end{align}

          We proceed to estimate $\mathcal{B}$. Denote 
          \begin{align}\label{G}
          G=\left(P_{t-s+r}\bm{1}_{[y\,, x]}^{(N)}- P_{r}\bm{1}_{[y\,, x]}^{(N)}
          \right)^2.
          \end{align}
          Again, applying the Burkholder's inequality and the Cauchy-Schwartz inequality,
          \begin{align}\label{B}
          \mathcal{B}&\leq 4 \E_{\lambda}\left[\left|\int_0^s \langle X_r\,, G\rangle \d r\right|^2\right]
           \leq 4s\int_0^s \E_{\lambda}\left[ \langle X_r\,, G\rangle^2\right]\d r \nonumber\\
           &=4s\int_0^s \d r\, \left( \langle \lambda\,, G\rangle^2 +\int_0^r \langle \lambda \,, (P_\theta G)^2\rangle \d \theta
             \right)\nonumber\\
            &= \mathcal{B}_1 + \mathcal{B}_2,
          \end{align}
          where the first equality holds by \eqref{secondmoment} and 
          \begin{align*}
          \mathcal{B}_1&=4s\int_0^s \langle \lambda\,, G\rangle^2\d r\\
          \mathcal{B}_2&=4s\int_0^s \d r\, \int_0^r \langle \lambda \,, (P_\theta G)^2\rangle \d \theta.
          \end{align*}
          We first estimate $\mathcal{B}_1$. According to \eqref{id:Plan},
          \begin{align}\label{lambdaG}
          \langle \lambda\,, G\rangle&=\left \langle \lambda\,,
           \left(P_{t-s+r}\bm{1}_{[y\,, x]}^{(N)} 
          \right)^2\right\rangle
          +\left \langle \lambda\,,
           \left( P_{r}\bm{1}_{[y\,, x]}^{(N)} 
          \right)^2\right\rangle
          -2\left \langle \lambda\,,
           P_{t-s+r}\bm{1}_{[y\,, x]}^{(N)}P_{r}\bm{1}_{[y\,, x]}^{(N)} 
          \right\rangle\nonumber\\
          &=\frac{N}{2\pi}\int_\R\left[
          \e^{-(t-s+r)z^2}+\e^{-rz^2}-2\e^{-((t-s)/2+r)z^2}\right]\left|\widehat{\bm{1}_{[y\,, x]}}(Nz)\right|^2\d z\nonumber\\
          &=\frac{N}{2\pi}\int_\R\e^{-rz^2}\left(1-
          \e^{-(t-s)z^2/2}\right)^2\left|\widehat{\bm{1}_{[y\,, x]}}(Nz)\right|^2\d z,
          \end{align}
          which implies that
          \begin{align}\label{lambdaGbound}
          \langle \lambda\,, G\rangle&\leq \frac{N}{2\pi}\int_\R\left|\widehat{\bm{1}_{[y\,, x]}}(Nz)\right|^2\d z=\|\bm{1}_{[y\,, x]}\|^2_{L^2(\R)}= |x-y|.
          \end{align}
          \textbf{Case 1}: $|x-y|\leq |t-s|^{1/2}$. In this case, we use \eqref{lambdaGbound}, \eqref{lambdaG}
           and the inequality 
          $1-\e^{-a}\leq a$ for all $a\geq 0$ to see that
          \begin{align}\label{B2case1}
          \mathcal{B}_1&\leq 4s|x-y|\int_0^s\langle \lambda\,, G\rangle \d r \nonumber\\
          & \leq \frac{4s|x-y|N}{4\pi}\int_0^s\d r\int_\R\e^{-rz^2} (t-s)z^2
          \left|\widehat{\bm{1}_{[y\,, x]}}(Nz)\right|^2\d z\nonumber\\
          & \leq \frac{4|t-s||x-y|N}{4\pi}\int_\R
          \left|\widehat{\bm{1}_{[y\,, x]}}(Nz)\right|^2\d z =2|t-s||x-y|^2 \leq 2|t-s|^{3/2}|x-y|^{3/2}.
          \end{align}
          \textbf{Case 2}: $|t-s|^{1/2}<|x-y|$.
          We observe that
          \begin{align}\label{fourier}
          \left|\widehat{\bm{1}_{[y\,, x]}}(a)\right|^2=\frac{2(1-\cos((x-y)a))}{a^2},\quad \text{for all $a\in\R$},
          \end{align}
          We appeal to  \eqref{lambdaGbound}, \eqref{lambdaG}, \eqref{fourier}
           and the inequality 
          $1-\e^{-a}\leq 1\wedge a$ for all $a\geq 0$
          to see that
          \begin{align}\label{B2case2}
          \mathcal{B}_1&\leq 4s|x-y|\int_0^s\langle \lambda\,, G\rangle \d r \nonumber\\
          &\leq  \frac{4|x-y|N}{2\pi} \int_{\R}\frac{(1-\e^{-sz^2})}{z^2}\left( 1\wedge [|t-s|^2z^4]\right)
          \frac{2(1-\cos((x-y)Nz)}{(Nz)^2}\d z\nonumber\\
          &\leq 4|x-y|\int_\R\frac{1\wedge [|t-s|^2z^4]}{z^4}\d z = \frac{32}3 |t-s|^{3/2}|x-y| \leq 
           \frac{32}3 |t-s|^{5/4}|x-y|^{5/4}.
          \end{align}
          We conclude from \eqref{B2case1} and \eqref{B2case2} that in both cases we have 
          \begin{align}\label{B2bound}
          \mathcal{B}_1\lesssim |t-s|^{5/4}|x-y|^{5/4}.
          \end{align}

          We next estimate $\mathcal{B}_2$. Recall the function $G$ defined in \eqref{G}. We apply the Plancherel's
          identity to see that
          \begin{align*}
          \langle \lambda\,, (P_\theta G)^2\rangle &= \frac{1}{2\pi}\int_\R\e^{-\theta z^2}|\hat{G}(z)|^2\d z\nonumber\\
          &\leq \frac{|\hat{G}(0)|^2}{2\pi}\int_\R\e^{-\theta z^2}\d z
          \asymp \theta^{-1/2}\langle \lambda\,, G\rangle^2 .
          \end{align*}
          Hence,
          \begin{align}\label{B1bound}
          \mathcal{B}_2&\lesssim 4s\int_0^s\sqrt{r}\langle\lambda\,, G\rangle^2 \d r \leq 4s\int_0^s
          \langle\lambda\,, G\rangle^2 \d r\nonumber\\
          &=\mathcal{B}_1\lesssim |t-s|^{5/4}|x-y|^{5/4},
          \end{align}
          where the last inequality follows from \eqref{B2bound}. 
          
          Therefore, the estimate \eqref{eq:holder} follows from \eqref{A+B}, \eqref{Abound2}, \eqref{B},
          \eqref{B2bound} and \eqref{B1bound}.
\end{proof}

\begin{remark}
          The exponent $\frac54$ (greater than 1) in \eqref{eq:holder} will be sufficient for tightness. It might not be optimal and
          one may expect this exponent can be replaced by $2$, as in \eqref{Abound2}.
\end{remark}

\begin{proposition}\label{margin:holder}
          For any $k\geq2$, there exists $C>0$ such that for all $(t, s, x, y)\in [0, 1]^4$ and $N\geq1$
          \begin{align}
          &\E\left[|V_N(t\,, x)-V_N(t\,,y)|^k\right] \leq C |x-y|^{k/2}, \label{space:holder}\\
          &\E\left[|V_N(t\,, x)-V_N(s\,,x)|^k\right] \leq C |t-s|^{k/2}. \label{time:holder}
          \end{align}
\end{proposition}
\begin{proof}
          We first prove \eqref{space:holder} and assume without loss generality $y\leq x$. According to \eqref{V_N} and Burkholder's inequality,
          \begin{align}
          \E\left[|V_N(t\,, x)-V_N(t\,,y)|^k\right] &= \E\left[\left|\int_0^t\int_\R \sqrt{u(r\,, z)} P_{t-r}\bm{1}_{[y,\, x]}^{(N)}(z)W(\d r\, \d z)\right|^k\right]\nonumber\\
          &\lesssim  \E\left[\left|\int_0^t\int_\R u(r\,, z) P^2_{t-r}\bm{1}_{[y,\, x]}^{(N)}(z)\d z\d r\right|^{k/2}\right]\nonumber\\
          &\leq  \left|\int_0^t\int_\R \|u(r\,, z)\|_{k/2} P^2_{t-r}\bm{1}_{[y,\, x]}^{(N)}(z)\d z\d r\right|^{k/2}\nonumber\\
          &\lesssim  \left|\int_0^t\int_\R P^2_{r}\bm{1}_{[y,\, x]}^{(N)}(z)\d z\d r\right|^{k/2}, \label{spacebound}
          \end{align}
           where the second inequality holds by Minkowski's inequality and the third by \eqref{momentbound}.  Now we apply \eqref{eq:asym2} to see that 
           for all $N\geq1$ and $(t, x, y)\in [0, 1]^3$
           \begin{align*}
           \int_0^t\int_\R P^2_{r}\bm{1}_{[y,\, x]}^{(N)}(z)\d z\d r \leq t\, \|\bm{1}_{[y\,, x]}\|^2_{L^2(\R)} \leq |x-y|,
           \end{align*}
           which together with \eqref{spacebound} yields \eqref{space:holder}.
           
           We proceed to prove \eqref{time:holder} and assume without loss generality $s\leq t$. We appeal to Burkholder's inequality and Minkowski's inequality again to see that
           \begin{align}\label{timebound}
           \E\left[|V_N(t\,, x)-V_N(s\,,x)|^k\right] & \lesssim \E\left[\left|\int_s^t\int_\R \sqrt{u(r\,, z)} P_{t-r}\bm{1}_{[0,\, x]}^{(N)}(z)W(\d r\, \d z)\right|^k\right] \nonumber\\
           & \quad + \E\left[\left|\int_0^s\int_\R \sqrt{u(r\,, z)} \left(P_{t-r}\bm{1}_{[0,\, x]}^{(N)}(z)- P_{t-r}\bm{1}_{[0,\, x]}^{(N)}(z)\right)W(\d r\, \d z)\right|^k\right] \nonumber\\
            & \lesssim \left|\int_s^t\int_\R \|u(r\,, z)\|_{k/2} P^2_{t-r}\bm{1}_{[0,\, x]}^{(N)}(z)\d z\d r\right|^{k/2} \nonumber\\
           & \quad + \left|\int_0^s\int_\R  \|u(r\,, z)\|_{k/2}  \left(P_{t-r}\bm{1}_{[0,\, x]}^{(N)}(z)- P_{t-r}\bm{1}_{[0,\, x]}^{(N)}(z)\right)^2\d z\d r\right|^{k/2} \nonumber\\
                       & \lesssim \left|\int_0^{t-s}\int_\R  P^2_{r}\bm{1}_{[0,\, x]}^{(N)}(z)\d z\d r\right|^{k/2} \nonumber\\
           & \quad + \left|\int_0^s\int_\R\left(P_{t-s+r}\bm{1}_{[0,\, x]}^{(N)}(z)- P_{r}\bm{1}_{[0,\, x]}^{(N)}(z)\right)^2\d z\d r\right|^{k/2},
           \end{align}
           thanks to \eqref{momentbound}. The estimate \eqref{eq:asym2} ensures that for all $N\geq 1$ and $(t, s, x)\in [0, 1]^3$
           \begin{align}\label{timebound1}
           \int_0^{t-s}\int_\R  P^2_{r}\bm{1}_{[0,\, x]}^{(N)}(z)\d z\d r \leq (t-s) \|\bm{1}_{[0, x]}\|^2_{L^2(\R)}\leq (t-s).
           \end{align}
           Moreover, from the calculation in \eqref{lambdaG} (with $y=0$), we have for all $(t, s, x)\in [0, 1]^3$
           \begin{align}\label{timebound2}
           &\int_0^s\int_\R\left(P_{t-s+r}\bm{1}_{[0,\, x]}^{(N)}(z)- P_{r}\bm{1}_{[0,\, x]}^{(N)}(z)\right)^2\d z\d r \nonumber\\
           &\quad =\frac{N}{2\pi}\int_\R \frac{1-\e^{-sz^2}}{z^2}\left(1-
          \e^{-(t-s)z^2/2}\right)^2\left|\widehat{\bm{1}_{[0\,, x]}}(Nz)\right|^2\d z \nonumber\\
          &\quad \leq \frac{(t-s)N}{4\pi}\int_\R\left|\widehat{\bm{1}_{[0\,, x]}}(Nz)\right|^2\d z = \frac{(t-s)x}{2} \leq (t-s),
                     \end{align}
                     where in the first inequality, we have used that $1-\e^{-a} \leq 1 \wedge a$ for all $a\geq 0$. Therefore, the estimate \eqref{time:holder} follows from \eqref{timebound}, \eqref{timebound1}
                     and \eqref{timebound2}. The proof is complete. 
\end{proof}

\section{Proof of Theorem \ref{th:main}}\label{th:proof}

\begin{proof}[Proof of Theorem \ref{th:main}]
Choose $0<t_m<\ldots < t_1\leq 1$ and $x_1, \ldots, x_m \in [0, 1]$. 
Recall the random variables $V_N(t_k\,, x_k), k=1, \ldots, m$ as defined in \eqref{V_N}.
By \eqref{secondmoment} and \eqref{eq:asym2}, 
\begin{align*}
\sup_{1\leq k\leq m}\sup_{N> 0}\E_{\lambda}\left[V_N(t_k\,, x_k)^2\right]<\infty,
\end{align*}
which implies that the family of random vectors $(V_N(t_1\,, x_1), \ldots,  V_N(t_m\,, x_m))_{N>0}$ are tight on $\R^m$.  Hence, for any sequence $(n_j)_{j=1}^{\infty}$ ($n_j\to\infty$ as $j\to\infty$), there exists a subsequence $(n'_{j})_{j=1}^{\infty}$ and a random vector $Y=(Y_1, \ldots, Y_m)$ such that as $j\to\infty$,
\begin{align*}
(V_{n'_j}(t_1\,, x_1), \ldots, V_{n'_j}(t_m\,, x_m)) \to (Y_1, \ldots, Y_m) \quad \text{in distribution}.
\end{align*}
Proposition \ref{fddcon} ensures that for all $\theta_1 \geq 0, \ldots, \theta_m\geq 0$, 
\begin{align*}
\lim_{j\to\infty}\E\left[\e^{-(\theta_1V_{n'_j}(t_1\,, x_1)+ \ldots + \theta_mV_{n'_j}(t_m\,, x_m))}\right]= \E\left[\e^{-(\theta_1W(t_1, x_1)+ \ldots + \theta_mW(t_m, x_m))}\right],
\end{align*}
where $W$ denotes the Brownian sheet. We can apply the same arguments as in \cite[p.106-107]{Isc86} to see that  for all $\theta_1 \geq 0, \ldots, \theta_m\geq 0$
\begin{align*}
 \E\left[\e^{-(\theta_1Y_1+ \ldots + \theta_mY_m)}\right]= \E\left[\e^{-(\theta_1W(t_1, x_1)+ \ldots + \theta_mW(t_m, x_m))}\right].
\end{align*}
By the analytic continuation theorem,  the random vector $(Y_1, \ldots, Y_m)$ has the same distribution as $(W(t_1, x_1), \ldots, W(t_m, x_m))$.  Therefore, we conclude that as $N\to\infty$, 
\begin{align*}
(V_{N}(t_1\,, x_1), \ldots, V_{N}(t_m\,, x_m)) \to (W(t_1, x_1), \ldots, W(t_m, x_m)) \quad \text{in distribution}.
\end{align*}
which implies that the finite-dimensional distributions of  $\left\{N^{-1/2}\int_0^{xN}[u(t\,, z)- 1]\, \d z\right\}_{(t, x)\in [0, 1]^2}$
converge to those of Brownian sheet. The tightness follows from Propositions \ref{prop:holder} and \ref{margin:holder}
 (see \cite[Theorem 3]{BiW71}, \cite[Theorem 1]{Jol88} or \cite[Exercise 1.4.19]{Kun90}).  The proof is complete. 
\end{proof}

\vskip1cm
\begin{small}
\noindent\textbf{Zenghu Li} and \textbf{Fei Pu}
Laboratory of Mathematics and Complex Systems,
School of Mathematical Sciences, Beijing Normal University, 100875, Beijing, China
 \\
Emails: \texttt{lizh@bnu.edu.cn} and \texttt{fei.pu@bnu.edu.cn}\\
\end{small}

\end{document}